\documentclass[11pt]{article}
\usepackage[mathvars]{brsheader} % Required for inserting images
\abbreviateBBMathFont

\newcommand{\xto}[1]{\xrightarrow{#1}}
\newcommand{\MRto}{\xrightarrow[\text{MR}]{}}

\newtheorem{innerreplemma}{Lemma}
\newenvironment{replemma}[1]
  {\renewcommand\theinnerreplemma{#1}\innerreplemma}
  {\endinnerreplemma}

\DeclareMathOperator{\MR}{MR}
\DeclareMathOperator{\R}{R}

\title{Products of simplices are canonically Ramsey}
\author{Benedict \textsc{Randall Shaw}\footnotemark[1]}
\date{July 2026}

\begin{document}

\maketitle

\renewcommand{\thefootnote}{\fnsymbol{footnote}}
\footnotetext[1]{\href{mailto:bwr26@cam.ac.uk}{bwr26@cam.ac.uk}, Department of Pure Mathematics and Mathematical Statistics (DPMMS), University of Cambridge, Wilberforce Road, Cambridge, CB3 0WA, United Kingdom}

\begin{abstract}
A set of points \(C \subset \BBR^n\) is called \emph{canonically Ramsey} if there is some set of points \(S\subset \BBR^{n'}\) such that any colouring of \(S\), using any number of colours, must contain either a monochromatic copy of \(C\) or a rainbow copy of \(C\). Mao, Ozeki, and Wang \cite{mao2022euclideangallairamseytheory} introduced this notion, showing that 30-60-90 triangles are canonically Ramsey. Since then, various other canonically Ramsey configurations have been identified. The author \cite{randallshaw2026cuboidscanonicallyramsey} showed that cuboids are canonically Ramsey, while Ge, Shu, Xu, and Yu \cite{ge2026simplicesexhibitcanonicalramsey} recently showed that simplices are canonically Ramsey. We extend both of these results, proving that all products of simplices are canonically Ramsey.
\end{abstract}

\section{Introduction}
We consider finite sets of points, or \emph{configurations}, in some Euclidean space \(\BBR^n\). For \(n<m\), we will identify \(\BBR^n\) with a subset of \(\BBR^m\), so that when discussing a configuration, we may disregard the dimension of the ambient space. We will also identify \(\BBR^n\times \BBR^m\) with \(\BBR^{n+m}\) in the usual way, and thus discuss the product of two configurations as a configuration in its own right.

We will be particularly concerned with \emph{nondegenerate simplices}, configurations which are affinely independent---that is to say, sets comprising \(n\) points which are not all contained in some \((n-2)\)-dimensional hyperplane. For ease of reading, we will use \emph{simplex} to mean a nondegenerate simplex.

For any \(n\) and sets of points \(S,C\subset \BBR^n\), we write
\[S\xto{r} C\]
if any \(r\)-colouring of \(S\) contains a \emph{monochromatic} copy of \(C\)---that is to say, a set of points congruent to \(C\) all of which are assigned the same colour. A compactness argument shows that \(\BBR^n\xto{r} C\) if and only if some finite set of points \(S\subset \BBR^n\) has \(S\xto{r} C\). We say a configuration \(C\) is \emph{Ramsey} if, for any \(r\), there is some \(n\) such that \(\BBR^n\xto{r} C\). This notion was introduced by Erd\H{o}s, Graham, Montgomery, Rothschild, Spencer, and Straus \cite{erdos1973}, who showed that all Ramsey sets are spherical, and that the family of Ramsey sets is closed under products---in particular, that it includes all \emph{cuboids}, sets of the form \(\{0,b_1\}\times\dots\times\{0,b_k\}\). Frankl and R\"odl \cite{frankl1990} later showed that simplices were Ramsey, in fact proving a stronger version of this notion. However, the question of exactly which sets are Ramsey remains open.

We may ask what happens, then, if we try to colour Euclidean space without any restriction on our set of colours. Certainly we cannot hope to force a monochromatic copy of any configuration of more than one point, as it would simply be possible to assign every point a different colour. Instead, we ask whether any colouring must contain either a monochromatic copy of our configuration, or a \emph{rainbow} copy of our configuration---that is to say, a set of points congruent to our configuration which all receive different colours. Thus for any set of points \(S\) and configuration \(C\), we write
\[S\MRto C\]
if every colouring of \(S\) contains either a monochromatic copy of \(C\) or a rainbow copy of \(C\). If there is some \(n\) such that \(\BBR^n\MRto C\), we say that \(C\) is \emph{canonically Ramsey}.

This notion was introduced by Mao, Ozeki, and Wang \cite{mao2022euclideangallairamseytheory}, who showed that the 30-60-90 triangle is canonically Ramsey. Cheng and Xu \cite{cheng2023euclideangallairamseyvariousconfigurations} showed that squares, right-angled triangles, and a large family of other triangles and simplices are canonically Ramsey. Geh\'er, Sagdeev, and T\'oth \cite{geher2026} showed that \emph{hypercubes}---sets of the form \(\{0,b\}^k\)---are canonically Ramsey. Fang, Ge, Shu, Xu, Xu, and Yang \cite{fang2025canonicalramseytrianglesrectangles} then showed that all triangles, and all rectangles, as well as a larger family of simplices with four points than previously known, are canonically Ramsey, and conjectured that all Ramsey sets should be canonically Ramsey.

The author \cite{randallshaw2026cuboidscanonicallyramsey} showed that all cuboids are canonically Ramsey, which itself implies that a further class of simplices are canonically Ramsey. Recently, Ge, Shu, Xu, and Yu \cite{ge2026simplicesexhibitcanonicalramsey} proved that in fact all simplices are canonically Ramsey. Our theorem extends both of these results:

\begin{theorem}\label{thm:prodsimpMR}
    For any \(k\), and any simplices \(S_1,\dots,S_k\), their product \(S_1\times\dots\times S_k\) is canonically Ramsey.
\end{theorem}

That is to say, products of simplices are canonically Ramsey. This result includes all currently known canonically Ramsey sets. Note that this also implies that all subsets of products of simplices are canonically Ramsey.

However, it is not yet known whether the entire family of canonically Ramsey sets is closed under products. This question is of some interest, as it would imply that all canonically Ramsey sets would be Ramsey for the following reason: any canonically Ramsey set \(C\) would be contained in arbitrarily large canonically Ramsey sets---that is to say, for any \(r\), some canonically Ramsey \(C'\) containing \(C\) would have \(|C'|>r\). In particular, then any \(r\)-colouring of a set \(S\) witnessing that \(C'\) is canonically Ramsey would have to contain a monochromatic copy of \(C'\), as \(r\) would not be enough colours for a rainbow copy of \(C'\). Thus we would have \(S\xto{r} C\), and since \(r\) was arbitrary, this would imply that \(C\) itself is Ramsey.

Before we outline our argument, we introduce a slightly more general notation, as follows: for a set of points \(S\), and configurations \(C,C'\), we write
\[S\MRto (C,C')\]
if any colouring of \(S\) admits either a monochromatic copy of \(C\), or a rainbow copy of \(C'\).

Our argument divides into three main stages. First, we show that for any simplices \(A,S_1,\dots,S_k\), there is some product of simplices \(C\) such that any colouring of \(C\) either admits a monochromatic copy of \(A\), or a \emph{properly coloured} copy of \(S_1\times\dots\times S_k\)---that is to say, one in which, for each \(i\), each copy of \(S_i\) of the form \(\{s_1\}\times\dots\times\{s_{i-1}\}\times S_i\times \{s_{i+1}\}\times\dots\times\{s_k\}\) is rainbow. This is achieved by iteratively applying the result of \cite{ge2026simplicesexhibitcanonicalramsey} that all simplices are canonically Ramsey.

Secondly, we show that for any simplices \(S_1,\dots,S_k\), there exist simplices \(S'_1,\dots,S'_k\) such that any proper colouring of \(S'_1\times\dots\times S'_k\) admits a monochromatic copy of \(S_1\times\dots\times S_k\). Here each \(S'_i\) will itself be a Ramsey witness for \(S_i\), and so it will be important to use the fact proven in \cite{ge2026simplicesexhibitcanonicalramsey} that we may choose Ramsey witnesses for simplices to be simplices themselves.

Finally, we note that the first two stages allow us to construct a product of simplices \(C\) such that
\[C\MRto (A,S_1\times\dots\times S_k).\]
We show that whenever \(C\MRto(A,Y)\) and \(C'\MRto(A',Y)\), for some \(A,A',C,C',Y\) such that \(C'\) is Ramsey, then we may construct a configuration \(C''\) such that \(C''\MRto(A\times A', Y)\). In particular, since the configuration \(C\) we constructed earlier in our argument is a product of simplices, it is certainly Ramsey. Applying this \(k-1\) times, we complete our proof.

\section{Overview}

We will require several known results, which we now introduce before describing our argument:
\begin{theorem}[\cite{frankl1990}]\label{thm:simplexR}
Every simplex is Ramsey.
\end{theorem}
Indeed, it is shown in \cite{ge2026simplicesexhibitcanonicalramsey} that we may assume that the configuration witnessing this is itself a simplex. Thus, for each simplex \(S\), let \(\R(S,r)\) be a simplex that witnesses this for \(r\)-colourings---so \(\R(S,r)\xto{r} S\).

\begin{theorem}[\cite{erdos1973}]\label{thm:productR}
    The product of two Ramsey sets is Ramsey.
\end{theorem}
In particular, we will use that any product of simplices is Ramsey. In fact, using the previous observation, it is possible to construct a witness that a product of simplices is Ramsey that is itself a product of simplices.

\begin{theorem}[\cite{ge2026simplicesexhibitcanonicalramsey}]\label{thm:simplicesMR}
For any two simplices \(S_1,S_2\), there is a simplex \(C\) such that
\[C\MRto (S_1,S_2).\]
\end{theorem}
Similarly, let \(\MR(S_1,S_2)\) be a simplex witnessing this---so \(\MR(S_1,S_2)\MRto (S_1,S_2)\). We will write \(\MR^a(S_1,S_2)\) to mean the result of repeatedly applying \(\MR(S_1,\cdot)\) to \(S_2\) \(a\) times. Note that for this to exist, it was important that \(\MR(S_1,S_2)\) could itself be taken to be a simplex.

We now describe our argument, the bulk of which is contained in three lemmas. Recall that a colouring of \(S_1\times\dots\times S_k\) is \emph{proper} if every subset of the form \(\{s_1\}\times\dots\times\{s_{i-1}\}\times S_i\times \{s_{i+1}\}\times\dots\times\{s_k\}\) is rainbow. We first prove the following lemma:

\begin{lemma}\label{lem:monovsproper}
For any simplices \(A,S_1,\dots,S_k\), there is some product of simplices \(C\) such that any colouring of \(C\) either admits a monochromatic copy of \(A\), or a copy of \(S_1\times\dots\times S_k\) whose colouring is proper.
\end{lemma}

Note that \(C\) is not just any configuration but a product of simplices---in particular, this implies that \(C\) is Ramsey, which we will use later on in the argument. The proof of this lemma inducts on \(k\), building \(C\) as the product of the configuration \(C'\) constructed by applying the inductive hypothesis to \(A,S_1,\dots,S_{k-1}\) with a large simplex \(D\). We will assume that our colouring of \(C\) has no monochromatic \(A\), so in particular, every copy of \(C'\) contains a properly coloured \(S_1\times\dots\times S_{k-1}\).

We choose \(D\) to be a Ramsey witness \(\R(D',r)\), for a suitable simplex \(D'\), and for \(r\) the number of copies of \(S_1\times\dots\times S_{k-1}\) in \(C'\). Then in fact a standard product argument, considering colouring \(D\) by the position of some properly coloured \(S_1\times\dots\times S_{k-1}\) in the associated copy of \(C'\), shows that the product \(C'\times D\) must in fact contain some copy of \(S_1\times\dots\times S_{k-1}\times D'\) in which each \(S_1\times\dots\times S_{k-1}\times \{m\}\) is properly coloured.

We now wish to restrict this to a copy of \(S_k\) within \(D'\) in which each set \(\{x\}\times S_k\) is rainbow. In fact, we do this for each \(x\in S_1\times\dots\times S_{k-1}\) in turn, restricting our set \(D'\) to a smaller simplex each time to guarantee that the copy of that simplex at \(x\) is rainbow. In order to do this, we choose \(D'\) to be of the form \(\MR^n(A,S_k)\), where \(n\) is the size of \(S_1\times\dots\times S_{k-1}\). At the \(i\)th step we restrict \(D'\) to a smaller set of the form \(\MR^{n-i}(A,S_k)\), stripping off one iteration of \(\MR(A,\cdot)\) each time. After \(n\) steps, we are left with a copy of \(S_k\), which corresponds to a properly coloured copy of \(S_1\times\dots\times S_k\).

We then show that this is enough to produce a rainbow product of simplices:

\begin{lemma}\label{lem:properimpliesrainbow}
    For any simplices \(S_1,\dots,S_k\), there exist simplices \(S'_1\dots,S'_k\) such that any proper colouring of \(S'_1\times\dots\times S'_k\) contains a rainbow copy of \(S_1\times\dots\times S_k\).
\end{lemma}

Again we induct on \(k\). Writing \(P=S_1\times\dots\times S_{k-1}\), we assume from our inductive hypothesis that there is some \(P'=S'_1\times\dots\times S'_{k-1}\) such that any proper colouring of \(P'\) contains a rainbow copy of \(P\). We will construct \(S'_k\) so that \(P'\times S'_k\) satisfies the statement of the lemma.

We will take \(S'_k\) to be a Ramsey witness \(\R(S_k,rm)\), for \(r\) the number of copies of \(P\) within \(P'\), and \(m=|P|^2-|P|+1\). Note that it is important here that we allowed our Ramsey witnesses to be simplices themselves for this. We then build an auxiliary colouring with \(rm\) colours on \(S'_k\) as follows: for each copy \(X\) of \(P\) within \(P'\), we build an auxiliary graph on the vertices of \(S'_k\) in which two vertices are joined if their corresponding copies of \(X\) share a colour. We show that this graph has maximum degree \(m-1\), and therefore has an \(m\)-colouring \(c_{G(X)}\).

For each point \(y\) of \(S'_k\), we note that \(P'\times\{y\}\) contains a rainbow copy of \(P\), and choose some such copy \(X\) of \(P\). We then assign \(y\) the colour \(\left(X,c_{G(X)}(y)\right)\). This is an \(rm\)-colouring, so there is some monochromatic copy of \(S_k\). But then this corresponds to some \(P\times S_k\) in which each copy of \(P\) is rainbow, and no two copies of \(P\) share a colour---hence this is itself a rainbow copy of \(P\times S_k\).

Finally, we prove a lemma which will allow us to move from finding monochromatic simplices to monochromatic products of simplices, using a product argument.

\begin{lemma}\label{lem:monoprod}
Suppose that for some configurations \(C,C',A,A',Y\), we have both \(C\MRto (A,Y)\) and \(C'\MRto(A',Y)\), and that additionally \(C'\) is Ramsey. Then there is a configuration \(C''\) such that \[C''\MRto(A\times A',Y).\]
\end{lemma}

We require \(C'\) to be Ramsey in order to choose \(C''\) of the form \(C\times D\), for \(D\) a suitable Ramsey witness \(\R(C',r)\), where \(r\) is the number of copies of \(A\) within \(C\). Consider a colouring of \(C''\) with no rainbow copy of \(Y\). We then follow a product argument, producing a copy of \(A\times C'\) in which all copies of \(C'\) have the same colouring. This gives a monochromatic copy of \(A\times A'\), as desired.

We then put these three lemmas together to prove Theorem \ref{thm:prodsimpMR}. Composing Lemmas \ref{lem:monovsproper} and \ref{lem:properimpliesrainbow} allows us to construct, for any simplices \(A,S_1,\dots,S_k\), a product of simplices \(C\) with 
\[C\MRto (A,S_1,\dots,S_k).\]
Since the \(C\) this generates is Ramsey, we may use Lemma \ref{lem:monoprod} to combine these configurations for \(A=S_1,\dots,S_k\) in turn. The resulting configuration will then witness that \(S_1\times\dots\times S_k\) is canonically Ramsey, as desired.

\section{Three important lemmas}
Recall that we say a colouring of \(S_1\times\dots\times S_k\) is \emph{proper} if any two points \((s_1,\dots,s_k),(s'_1,\dots,s'_k)\in S_1\times\dots\times S_k\) that have \(s_i\neq s'_i\) for exactly one \(i\) are differently coloured---equivalently, every copy of \(S_i\) with the form \(\{s_1\}\times\dots\times\{s_{i-1}\}\times S_i\times \{s_{i+1}\}\times\dots\times\{s_k\}\) is rainbow. We first prove Lemma \ref{lem:monovsproper}, which we now restate:
\begin{replemma}{\ref{lem:monovsproper}}
For any simplices \(A,S_1,\dots,S_k\), there is some product of simplices \(C\) such that any colouring of \(C\) either admits a monochromatic copy of \(A\), or a copy of \(S_1\times\dots\times S_k\) whose colouring is proper.
\end{replemma}
\begin{proof}
    We prove this by induction on \(k\). Notice that the base case \(k=1\) is simply Theorem \ref{thm:simplicesMR}. Henceforth, we may assume that \(k\geq 2\), and that the theorem holds for \(k-1\). Let \(C'\) be the configuration given by the inductive hypothesis, so that any colouring of \(C'\) contains either a monochromatic copy of \(A\) or a properly coloured copy of \(S_1\times\dots\times S_{k-1}\).
    
    Recall that we write \(\MR^a(A,X)\) to mean the result of applying \(\MR(A,\cdot)\) to \(X\) \(a\) times---so for example, \(\MR^2(A,X)=\MR(A,\MR(A,X))\). Let \(r\) be the number of copies of \(S_1\times\dots\times S_{k-1}\) in \(C'\), let \(n=|S_1\times\dots\times S_{k-1}|\), and write \[D=\R\left(\MR^n(A,S_k),r\right).\]
    Our final configuration will be \(C=C'\times D\). Note that we made use of the fact that \(\MR^a(\cdot,\cdot)\) is still a simplex, and therefore Ramsey. In fact, by \cite{ge2026simplicesexhibitcanonicalramsey}, we assume \(D\) is a simplex, so that our final \(C\) will itself be a product of simplices.

    We now fix some colouring of \(C\), which we assume contains no monochromatic copy of \(A\). It will suffice to prove this contains a properly coloured copy of \(S_1\times\dots\times S_k\). From our inductive hypothesis, each copy of \(C'\) in the product \(C=C'\times D\) contains a properly coloured copy of \(S_1\times \dots \times S_{k-1}\). We build an auxiliary colouring of \(D\) which colours each \(d\in D\) by the set of vertices of some properly coloured copy of \(S_1\times \dots\times S_{k-1}\) in \(C'\times \{d\}\). Now by the choice of \(r\), this is an \(r\)-colouring of \(D\). Hence by the definition of \(D\), there is some copy \(D'\) of \(\MR^n(A,S_k)\) in \(D\) which is monochromatic in this auxiliary colouring---that is to say, there is a fixed copy \(X\) of \(S_1\times\dots\times S_{k-1}\) such that every set of the form \(X\times\{d\}\) for \(d\in D'\) is already properly coloured.

    Now let the points of \(X\) be \(x_1,\dots,x_n\) in some order. Let \(D'_0=D'\), and notice that \(D'_0\) is a copy of \(\MR^n(A,S_k)\). We now choose \(D'_1,\dots,D'_n\) sequentially such that each \(D'_i\) is a copy of \(\MR^{n-i}(A,S_k)\), and within \(D'_i\), for each \(j\leq i\), the set \(\{x_j\}\times D'_i\) is rainbow. It will suffice to choose \(D'_i\subset D'_{i-1}\) such that \(\{x_i\}\times D'_i\) is rainbow. But by construction, we know \(D'_{i-1}\) is a copy of
    \[\MR\left(A,\MR^{n-i}(A,S_k)\right).\]
    But since we have assumed that \(C\) has no monochromatic copy of \(A\), certainly \(\{x_i\} \times D'_{i-1}\) contains some rainbow copy of \(\MR^{n-i}(A,S_k)\). Take this to be \(\{x_i\}\times D'_i\).

    Now \(D'_n\) is a copy of \(S_k\) such that each set of the form \(\{x\}\times D'_n\), for \(x\in X\), is rainbow. But then \(X\times D'_n\) is a properly coloured copy of \(S_1\times\dots\times S_k\), as desired.
\end{proof}

We now show that being able to construct properly coloured products of simplices is sufficient to construct rainbow products of simplices, proving Lemma \ref{lem:properimpliesrainbow}, which we now restate:
\begin{replemma}{\ref{lem:properimpliesrainbow}}
    For any simplices \(S_1,\dots,S_k\), there exist simplices \(S'_1\dots,S'_k\) such that any proper colouring of \(S'_1\times\dots\times S'_k\) contains a rainbow copy of \(S_1\times\dots\times S_k\).
\end{replemma}
\begin{proof}
    We prove this by induction on \(k\). For \(k=1\), we simply take \(S'_1=S_1\), as proper colourings of a single simplex are exactly rainbow colourings. We may therefore assume that \(k\geq 2\), and that the theorem holds for \(k-1\). Let \(S'_1,\dots,S'_{k-1}\) be the simplices given by the inductive hypothesis applied to \(S_1,\dots,S_{k-1}\), and let \(P=S_1\times\dots\times S_{k-1}\), and \(P'=S'_1\times\dots\times S'_{k-1}\).

    Let \(r\) be the number of copies of \(P\) in \(P'\), and recall that for any simplex \(S\), there is some simplex \(\R(S,r)\) such that any \(r\)-colouring of \(\R(S,r)\) contains a monochromatic \(r\). Write \(m=|P|^2-|P|+1\), and let \(S'_k=\R\left(S_k,rm\right)\).

    Given a proper colouring of \(P'\times S'_k\), we now find a rainbow copy of \(P\times S_k\) within it. For each copy \(X\) of \(P\) in \(P'\), consider for each point \(y\in S'_k\) the set of colours assigned to \(X\times \{y\}\), and build an auxiliary graph \(G(X)\) on the vertices of \(S'_k\) where \(y\) and \(y'\) are joined by an edge if some colour appears in both \(X\times \{y\}\) and \(X\times \{y'\}\).
    
    Since \(P'\times S'_k\) is properly coloured, each set \(\{x\}\times S'_k\) is rainbow, and so each colour appears at most \(|P|\) times in \(X\times S'_k\). But now each colour of \(X\times \{y\}\) appears in at most \(|P|-1\) other such sets, and so since there are at most \(|P|\) colours in \(X\times \{y\}\), the graph \(G(X)\) has maximum degree at most \(|P|(|P|-1)\). It therefore has a proper vertex colouring \(c_{G(X)}\) with at most \(m\) colours.

    We now build an auxiliary colouring of \(S'_k\) as follows: for each point \(y\) of \(S'_k\), choose some rainbow copy \(X\) of \(P\) within \(P'\times \{y\}\), and let \(y\) be coloured by the ordered pair \(\left(X,c_{G(X)}(y)\right)\). Certainly this uses at most \(rm\) colours. Hence there is some copy \(Y\) of \(S_k\) that is monochromatic under this auxiliary colouring, say with colour \((X,i)\). But now for each \(y\in Y\), \(X\times \{y\}\) is rainbow, and since \(c^{-1}_{G(X)}(i)\) is an independent set of \(G(X)\), the sets of colours used in each \(X\times \{y\}\) are disjoint. Hence \(X\times Y\) is a rainbow copy of \(P\times S_k\), as desired.
\end{proof}

Finally, we restate and prove Lemma \ref{lem:monoprod}, allowing us to seek monochromatic copies of products of simplices.

\begin{replemma}{\ref{lem:monoprod}}
Suppose that for some configurations \(C,C',A,A',Y\), we have both \(C\MRto (A,Y)\) and \(C'\MRto(A',Y)\), and that additionally \(C'\) is Ramsey. Then there is a configuration \(C''\) such that \[C''\MRto(A\times A',Y).\]
\end{replemma}
\begin{proof}
    Let \(r\) be the number of copies of \(A\) in \(C\), and let \(D\) be some configuration such that \(D\xto{r}C'\). We write \(C''=C\times D\), and show that any colouring of \(C''\) that does not contain a rainbow copy of \(Y\) must contain a monochromatic copy of \(A\times A'\). Fix some colouring of \(C\times D\) with no rainbow copy of \(Y\), and notice that each set with form \(C\times \{d\}\) must contain a monochromatic copy of \(A\). Consider the auxiliary colouring of \(D\) that colours each \(d\in D\) by the position of that monochromatic copy of \(A\) within \(C\). This is certainly an \(r\)-colouring, so there must be some copy \(D'\) of \(C'\) within \(D\) which is monochromatic under this colouring---that is to say, there is some copy \(X\) of \(A\) within \(C\) such that each set \(X\times \{d\}\), with \(d\in D'\), is monochromatic
    
    But now every set of the form \(\{x\}\times D'\), with \(x\in X\), receives exactly the same colouring. We have already assumed this does not contain a rainbow copy of \(Y\), so as \(D'\) is a copy of \(C'\), it must contain a monochromatic copy \(X'\) of \(A'\). But now \(X\times X'\) is a monochromatic copy of \(A\times A'\), giving the desired result.
\end{proof}

Recall that we may choose a Ramsey witness for a product of simplices to be a product of simplices itself. We therefore note that in this lemma, if \(C\) and \(C'\) are both products of simplices, then we may in fact choose \(C''\) also to be a product of simplices.

\section{Proof of Theorem \ref{thm:prodsimpMR}}
We now prove our main theorem.
\begin{proof}[Proof of Theorem \ref{thm:prodsimpMR}]
    We prove that for any simplices \(S_1,\dots,S_k\), the product \(S_1,\dots, S_k\) is canonically Ramsey. First we show that that for each simplex \(A\), there is some product of simplices \(C(A)\) such that
    \[C(A)\MRto\left(A,S_1\times\dots\times S_k\right).\]
    By Lemma \ref{lem:properimpliesrainbow}, there are simplices \(S'_1,\dots,S'_k\) such that any proper colouring of \(S'_1,\times\dots\times S'_k\) contains a rainbow copy of \(S_1\times\dots\times S_k\). But now applying Lemma \ref{lem:monovsproper} to \(A, S'_1,\dots,S'_k\), we obtain a product of simplices \(C(A)\) such that any colouring of \(C(A)\) without a monochromatic copy of \(A\) must have a properly coloured copy of \(S'_1,\times\dots\times S'_k\)---but then this contains a rainbow copy of \(S_1\times\dots\times S_k\), as desired.

    Now let \(C_1=C\left(S_1\right)\). We now build configurations \(C_i\) with \[C_i\MRto \left(S_1\times\dots\times S_i,S_1\times\dots\times S_k\right).\]
    But we can simply do this by applying Lemma \ref{lem:monoprod} to \(C_{i-1}\) and \(C(S_i)\) in sequence, noting that \(C(S_i)\) is Ramsey because it is a product of simplices. Then certainly we have
    \[C_k\MRto(S_1\times\dots\times S_k,S_1\times\dots\times S_k),\]
    witnessing the fact that \(S_1\times\dots\times S_k\) is canonically Ramsey.
\end{proof}

Indeed, since each \(C(A)\) is a product of simplices, it is possible to choose our witness \(C_k\) to be a product of simplices itself.

\section{Concluding remarks}
Fang, Ge, Shu, Xu, Xu, and Yang \cite{fang2025canonicalramseytrianglesrectangles} conjectured that all Ramsey sets should also be canonically Ramsey. Certainly Theorem \ref{thm:prodsimpMR} shows that a wide class of Ramsey sets are also canonically Ramsey. The next interesting positive case to check, therefore, is the following:
\begin{question}\label{q:pentagon}
    Is the set of vertices of a regular pentagon canonically Ramsey?
\end{question}
We note that the octahedron is already known to be canonically Ramsey, as it is contained in a \(4\)-dimensional hypercube and thus follows from the earlier result of Geh\'er, Sagdeev, and T\'oth \cite{geher2026}. Hence to prove any further platonic solids are canonically Ramsey requires an answer to Question \ref{q:pentagon}. Other regular cross-polytopes can be found similarly in large hypercubes, and therefore are also canonically Ramsey.

Conversely, the author asked in \cite{randallshaw2026cuboidscanonicallyramsey} whether all canonically Ramsey sets are Ramsey---that is, whether there is some set that is canonically Ramsey but not Ramsey. This question remains open. It was motivated by the fact that any set \(C\) contained in arbitrarily large canonically Ramsey sets \(C'\) must itself be Ramsey, as any set \(S\) witnessing that \(C'\) is canonically Ramsey must certainly have \(S\xto{r}C'\) for all \(r<|C'|\). Hence if, for example, the family of canonically Ramsey sets were closed under products, then all canonically Ramsey sets would be Ramsey. Indeed, even a result allowing us to add a single point to a canonically Ramsey set to create a new one would entail that all canonically Ramsey sets were Ramsey. Any such result would be of great interest.

There are sets known not to be canonically Ramsey---a spherical colouring with only three colours by Erd\H{o}s, Graham, Montgomery, Rothschild, Spencer, and Straus \cite{erdos1973} shows that the set \(\{0,1,2,3\}\) is not canonically Ramsey. The question of whether the set \(\{0,1,2\}\) is canonically Ramsey remains open---if it were, it would show that not all canonically Ramsey sets are Ramsey, and thus rule out any possible results allowing us to generate larger canonically Ramsey sets around existing ones. However, a colouring of \(\BBR^n\) containing neither monochromatic nor rainbow copies of \(\{0,1,2\}\) would have to be very carefully structured---in particular, Cheng and Xu \cite{cheng2023euclideangallairamseyvariousconfigurations} note that it could not be spherical. Thus a resolution of whether \(\{0,1,2\}\) is canonically Ramsey would be highly interesting either way.

\section{Acknowledgement}
The author is funded by an Internal Graduate Studentship of Trinity College, Cambridge. We thank Gaia Carenini for many helpful comments, and in particular for suggesting a better method for Lemma \ref{lem:properimpliesrainbow}.

	\bibliographystyle{plain}
\bibliography{main}

\end{document}